\documentclass[12pt,leqno]{article}
\usepackage{amstext}
\usepackage{theorem}
\usepackage{amssymb}
\usepackage{latexsym}

\usepackage{amsmath}
\usepackage{amscd}

\usepackage{color}

\usepackage[colorlinks=true,linkcolor=blue,citecolor=red]{hyperref}

\numberwithin{equation}{section}

\renewcommand{\epsilon}{\varepsilon}

\newcommand{\M}{\mathcal{M}}
\newcommand{\C}{\mathbb{C}}
\newcommand{\E}{\mathcal{E}}

\renewcommand{\H}{\mathcal{H}}

\newcommand{\Q}{\mathbb{Q}}
\newcommand{\R}{\mathbb{R}}

\newcommand{\Z}{\mathbb{Z}}

\DeclareMathOperator{\ord}{ord}

\DeclareMathOperator{\eval}{Eval}

\newcommand{\twp}{{\tilde\wp}}

\newcommand{\tbf}[1]{\textbf{#1}}

\newtheorem{theorem}{Theorem}
\newtheorem{lemma}[theorem]{Lemma}

\newtheorem{corollary}[theorem]{Corollary}
\newtheorem{question}{Question}
\newtheorem{remark}[theorem]{Remark}

\newtheorem{notation}[theorem]{Notation}

\newenvironment{proof}{\par\noindent\tbf{Proof \hspace{0.1cm}}}{\vspace{0.2cm}\hfill
\rule{2mm}{2mm}\par}

\title{Hilbert's tenth problem for complex meromorphic functions in several variables}
\author{Thanases Pheidas and Xavier Vidaux\footnote{This work has been partially financed by the first author's European Marie Curie Individual Fellowship
MCFI-2002-00722 and Chilean Fondecyt projects 1060947, 1090233 and 1130134, by the Hausdorff Institute of Mathematics (2009 Haudorff Trimester Program on Diophantine Equations), by the Universidad de Concepci\'on, and by the University of Crete-Heraklion.}}

\begin{document}

\date{}

\maketitle

\begin{abstract}
We prove an analogue of Hilbert's Tenth Problem for complex meromorphic functions. More precisely, we prove that the set of integers is positive existentially definable in fields of complex meromorphic functions in several variables over the language of rings, together with constant symbols for two of the independent variables and the set of constants, a unary relation symbol for non-zero functions, and a unary relation symbol for evaluation at a fixed point (a place). We obtain a similar result for analytic functions, where the place appears in the language as a binary predicate. In both cases, we only require the functions to be meromorphic (or analytic) on a set containing $\C$ in one of the variables (it can be germs in all the other variables). 
\end{abstract}


\tableofcontents

\section{Introduction}

Hilbert's Tenth Problem (H10) asked for an algorithm to decide whether or not an arbitrary system of polynomial equations over the integers has a solution in the integers. Based on work of M. Davis, H. Putnam and J. Robinson \cite{DPR61}, Matijasevich \cite{Matijasevich70} proved that such an algorithm does not exist: the positive existential theory of the ring of integers is undecidable.

In this work, we prove undecidability results for analogous problems for rings of complex meromorphic functions in several variables (at least $2$). Let $\bar z=(z_1,z_2)$. Let $\H_{\bar z}$ denote the ring of complex entire functions, and $\M_{\bar z}$ its field of fractions. In $\M_{\bar z}$, $\eval(f)$ stands for ``the function $f$ is well defined at $z_1=0$, and when evaluated at $z_1=0$, it is an analytic function of $z_2$ which takes the value $0$ at $z_2=0$'' --- note that it coincides with the usual concept of \emph{place}, as defined for instance in \cite[p. 349, and Example 4 p. 350]{LangAlgebra}. In $\H_{\bar z}$, $\eval_0(a,b)$ stands for ``$b\ne0$ and $\eval(a/b)$''. 

\begin{theorem}\label{mainLogicA}
For $\bar z=(z_1,z_2)$, the set $\Z$ of rational integers is positive-existentially definable in:
\begin{enumerate}
\item $\M_{\bar z}$, over the language of rings, together with a constant symbol for each variable, and with the unary predicate $\eval$. 
\item $\H_{\bar z}$, over the language of rings, together with a constant symbol for each variable, and with the binary predicate $\eval_0$.
\end{enumerate} 
Consequently, the positive-existential theory of each of these structures is undecidable.
\end{theorem}

In algorithmic terms, this means that there is no algorithm to decide whether an arbitrary system of polynomial equations with coefficients in $\Z[z_1,z_2]$, together with conditions of the form $\eval(f)$ (resp. $\eval_0(a,b)$) on some of the variables, has or does not have a solution in $\M_{\bar z}$ (resp. $\H_{\bar z}$). 

Problems of this kind occur often in applications of Mathematics, especially in the form of asking whether algebraic conditions that were obtained as a result of solving a differential equations with ``initial conditions'' (such as the ones expressed by the predicate $\eval$) represent global analytic or meromorphic functions or not.

Theorem \ref{mainLogicA} is actually a particular case of our main result, Theorem \ref{mainLogic}, below. Before stating this result, let us first discuss the existing literature on decidability problems related to holomorphic and meromorphic functions. 

In 1978, J. Denef \cite{Denef78,Denef79} gave the first H10-like results for rings of functions. In particular, he proved that for any domain $R$, H10 for rings of polynomials $R[z]$ over the language $L_z=L_r\cup\{z\}$, where $L_r=\{=,0,1,+,\cdot\}$ denotes the language of rings, is unsolvable (see the survey \cite[Section 1.2, Polynomial Rings]{PheidasZahidi00} for a discussion on why one needs to enrich the language of rings in order to obtain non trivial analogues of H10 for rings of functions). This means that there is no algorithm to decide whether or not an arbitrary system of polynomial equations over $\Z[z]$ has a solution over $R[z]$. 

Denef \cite{Denef78} also proved the unsolvability of H10 for some fields of rational functions $F(z)$, e.g. for $\R(z)$. Since then many more similar results have been produced (see for example \cite{Pheidas91,Videla94,Zahidi03}), all claiming undecidability for fields of rational functions, but a general theorem for {\it any} field of rational functions is still missing (though expected  by the experts to be true). In particular it is unknown whether the positive existential theory (or, even the full first order theory) of the field of rational functions $\C (z)$ of the variable $z$ over the complex numbers is decidable or undecidable (over the language $L_z$). A main obstacle to proving this has been that there is no known diophantine definition (or even a first order one) of \emph{order}: the property for a rational function $x\in \C(z)$ to take the value $0$ at $z=0$. The problem of defining the order was one of the main issue to solve H10 for rational functions over a finite field \cite{Pheidas91}. Another exception is the fact, proven in \cite{KimRoush92}, that H10 for $\C(z_1,z_2)$ is unsolvable over the language $L_{z_1,z_2}=L_r\cup\{z_1,z_2\}$. 
 
Though an impressive literature exists about \emph{algebraic extensions} of polynomial rings and rational function fields (see \cite{Shlapentokh07} and the references therein), little is known about subrings of \emph{completions} of these structures. Indeed, the only structures of analytic or meromorphic functions for which H10 is known to be unsolvable are the following:

\begin{enumerate}
\item The ring of entire functions over a non-Archimedean complete algebraically closed field of characteristic $0$, seen as an $L_z$-structure --- due to L. Lipshitz and the first author, \cite{LipshitzPheidas95}.
\item The ring of entire functions over a non-Archimedean complete algebraically closed field of positive characteristic, seen as an $L_z$-structure --- due to N. Garcia-Fritz and H. Pasten, see \cite{GarciaFritzPasten15}. 
\item The field of functions which are meromorphic over a non-Archimedean complete algebraically closed field of characteristic $0$, seen as an $L_{z,\ord}$-structure, where $\ord$ is the set of functions that take the value $0$ at $0$ --- due to the second author, see \cite{Vidaux03}.
\item The field of functions which are meromorphic over a complete algebraically closed field of odd characteristic, seen as an $L_{z,\ord}$-structure, where $\ord$ is the set of functions that take the value $0$ at $0$ --- due to H. Pasten, see \cite{Pasten16}.
\end{enumerate}
Note that all these results are for \emph{non-Archimedean} structures. 

For relevant questions and results, see also \cite{DenefGromov85,Rubel95,PheidasZahidi08}. We should emphasize that H10 for the ring $\H_z$ of complex entire functions in the variable $z$ is not known to be solvable or not over the language $L_z$ (the reference \cite{Pheidas95}, where a result is claimed, has a mistake which is discussed in \cite[Section 8]{PheidasZahidi00}).

As far as the full theory is concerned, R. Robinson \cite{RRobinson51} proved that the first order theory of the ring $\H_z$ is undecidable over $L_r$ (and hence also for rings $\H_{\bar z}$, where $\bar z$ is a tuple of variables), and recently H. Pasten \cite{Pasten16} proved the analogous result for fields of meromorphic functions in positive characteristic.

Before we state our main theorem, we need to introduce some notation. Given $m\ge2$ an integer and a non-empty connected subset  $B$ of $\C^m$, we will denote by $\M_{\bar z} (B)$ the field of meromorphic functions on an open superset of $B$ in the variables $\bar z=(z_1,\dots,z_m)$. We will see $\M_{\bar z} (B)$ as an $L_{z_1,z_2,\eval,C}$-structure, where 
$$
L_{z_1,z_2,\eval,C}=L_r\cup\{z_1,z_2,\eval,C\}
$$ 
and 
\begin{itemize}
\item $\eval$ is the set of functions $g$ which, when evaluated at $z_1=0$, are analytic functions in $z_2$ which take the value $0$ at $z_2=0$; and
\item $C$ is the set of constant functions.
\end{itemize} 
If $R$ is a subring of $\M_{\bar z} (B)$, we will feel free to write $\eval(a,b)$ to mean ``$b\ne0$ and $\eval(a/b)$'' and we will consider $R$ as an $L_{z_1,z_2,\eval,C,\ne}$-structure, where $L_{z_1,z_2,\eval,C,\ne}=L_{z_1,z_2,\eval,C}\cup\{\ne\}$, and $\ne$ stands for the set of non-zero functions. If $B=\C^m$, then $R$ will be considered as an $L_{z_1,z_2,\eval,\ne}$-structure, where $L_{z_1,z_2,\eval,\ne}=L_r\cup\{z_1,z_2,\eval,\ne\}$. Finally, for any set $B$ as above, we will write $\H_{\bar z}(B)$ for the subring of all analytic functions in $\M_{\bar z}(B)$. 

We can now state our main theorem. 

\begin{theorem}\label{mainLogic}
Let $m\ge 2$ be an integer. Let $B$ be a connected subset of $\C^m$ which contains $\C\times\{0\}\times D_3\dots\times D_m$, where each $D_i$ is a non-empty connected subset of $\C$. Let $R$ be a subring of $\M_{\bar z} (B)$ containing $\C [z_1,z_2]$. 
\begin{enumerate}
\item The set $\Z$ of rational integers is positive-existentially definable in $R$, seen as an $L_{z_1,z_2,\eval,C,\ne}$-structure. Consequently, the positive-existential $L_{z_1,z_2,\eval,C,\ne}$-theory of $R$ is undecidable.
\item Assume $B=\C^m$. The set $\Z$ of rational integers is positive-existentially definable in $R$, considered as an $L_{z_1,z_2,\eval,\ne}$-structure. Consequently, the  positive-existential $L_{z_1,z_2,\eval,\ne}$-theory of $R$ is undecidable.
\end{enumerate} 
\end{theorem}

Observe that the symbol $\ne$ can be easily removed for $R$ equal to either $\M_{\bar z}(B)$ or $\H_{\bar z}(B)$.

We view this theorem as evidence for the unsolvability of H10 for the ring of analytic functions in one variable $z$ over the language $L_z$, as all but one of the $D_i$ may be just singletons and, in the one variable case, the predicate $\eval$ corresponds to the classical valuation, which is easily definable in a positive existential way (recall that the ring of germs of analytic functions in one variable has a decidable theory --- see \cite{Kochen75}). This is the first result in the bibliography which proves undecidability for a diophantine problem in a ring of complex analytic functions over a mild extension of the ring language. 

H. Pasten pointed out to us Lemma 3.8 in \cite{BuzzardLu00}, which indicates that the usual way to get the integers from endomorphisms of elliptic curves should not work in the setting of complex meromorphic functions in one variable. Our method of proof is indeed a deviation from all proofs of similar results in the past, and goes as follows: we produce a rational function, defined on a concrete variety, which, although it can take up an uncountable number of values, has the property that, at some pre-determined point, whenever defined (in the sense of the predicate $\eval$), takes on values that are rational integers. And we show that all rational integers are obtainable in this way.

Here are some obvious questions that result from our work: 

\begin{question} Is $\eval$ positive-existentially definable over $\H_{\bar z}(\C^m)$, seen as an $L_{z_1,z_2}$-structure?
\end{question}

\begin{question} Is $\eval$ positive-existentially definable over $\M_{\bar z}(\C^m)$, seen as an $L_{z_1,z_2,\ord}$-structure, where $\ord(x)$ stands for ``the meromorphic function $x$ is analytic at $\bar z=(0,\dots,0)$ and takes the value $0$ at this point''?
\end{question}

We are in debt to Christos Kourouniotis, who helped us on an early version of this work. During the 15 years or so that we have been working on this project, we have had helpful discussions on some technical issues with Gustavo Avello J., Jan Denef, Antonio Laface, Leonard Lipshitz, Pavao Mardesic, Hector Pasten and Joseph H. Silverman. We are very grateful to each of them.

\section{Sketch of proof and Notation}\label{mainIdea}

Whenever $g$ is a function of the variables $(z_1,\dots,z_m)$, we will write $g_{z_i}$ for the partial derivative of $g$ with respect to $z_i$. We will work over $\M=\M_{z,\delta}(B)$, in the pair of independent variables $(z,\delta)$, where $B$ is a connected superset of $\C\times\{-2\}$. All equality symbols are interpreted over $\M$ (are functional equalities), except if stated otherwise. We will denote by $\H$ the ring of analytic functions in $\M$. 
   
Write $f(\delta,z)=z^3+\delta z^2+z$. Consider a solution $(x,y)\in \M^2$ of the ``Manin-Denef Equation'' $f(\delta,z)y^2=f(\delta,x)$, namely,
\begin{equation}\tag{MD}\label{MD}
(z^3+\delta z^2+z) y^2=x^3+\delta x^2+x.
\end{equation}

When $y$ is not the zero function, we will write 
$$
A_{xy}=\frac{x_z}{y}
\qquad\textrm{and}\qquad
\alpha_{xy}=\frac{x-1}{(z-1)y}
$$
and consider these expressions as elements of $\M$. Also we will write 
$$
\tilde A_{xy}=\left . A_{xy}\right|_{\delta=-2}
$$
and consider it as a meromorphic function in the variable $z$. 

First we will prove that $A_{xy}$ is analytic on $B$ (see Lemma \ref{xavier1}), so in particular $\tilde A_{xy}$ is a well defined analytic function on $\C$. In general $\alpha_{xy}$ may be  not continuous at $(z,\delta)=(1,-2)$, that is, $\alpha_{xy}$ evaluated at $z=1$ and then at $\delta =-2$, may either be undefined or may have a different value if evaluated in the reverse order. Still, using the fact that $A_{xy}$ is analytic, we will show that, whenever $\left . \alpha_{xy}\right|_{z=1,\delta=-2}$ is defined (and this happens often enough for our purposes), we have 
$$
\left . \ell\cdot\alpha_{xy}\right|_{z=1,\delta=-2}=\left . A_{xy}\right|_{z=1,\delta=-2}
$$ 
where $\ell\in\{1,-2\}$ (see Lemma \ref{alphaA}). 

In Section \ref{secrep} (Lemma \ref{central}) we prove that there are unique constants $\beta$ and $\gamma$, and a unique analytic function $h$ in the variable $z$, such that
\begin{equation}\label{diff}
\tilde A_{xy}=\beta +\gamma (z-1)+ \tilde fh_z+\frac{1}{2}\tilde f_z h,
\end{equation} 
where $\tilde f(z)=f(-2,z)=z(z-1)^2$. In particular, we have 
$$
\left . \tilde A_{xy}\right|_{z=1}=\beta.
$$ 
Our next task is to show that the constant $\beta$ is a rational integer (Lemma \ref{intvallem}). We do this in Sections \ref{secuni} and \ref{intval} in a way that we will now describe. 

We make the following uniformisation (topological) argument. First we observe that there is a periodic function $\tilde \wp (u)$ of the variable $u$, of period $2\pi i$, which is meromorphic in $u$, such that
$$
(\tilde \wp' )^2= \twp\cdot (\twp-1)^2=\tilde f\circ \tilde \wp
$$
(derivative with respect to $u$). We also find a meromorphic function $\tilde \xi (u)$ of $u$ with the property $\tilde \xi' (u)=\tilde \wp(u) -1$. 
  
We set 
\begin{equation}\label{uni2}
G(u) =\beta u+\gamma \tilde \xi (u)+\tilde \wp '(u) \cdot h(\tilde \wp (u))
\end{equation}
where $\beta$, $\gamma$ and $h$ are defined by Equation \eqref{diff}. Set $\tilde x=\left . x\right|_{\delta =-2}$ and 
$\tilde y=\left. x\right|_{\delta =-2}$, whenever it makes sense. Setting $z=\tilde \wp (u)$, we have $G'(u)=\tilde A_{xy}$. Using this fact, we compare the pairs of functions 
$(\tilde x(\tilde \wp (u)), \tilde \wp '(u)\cdot \tilde y (\tilde \wp (u)))$
and 
$(\tilde \wp (G), \tilde \wp' (G))$ and we establish that they are \emph{essentially} equal locally, hence globally --- see Lemma \ref{uniform}. In order to conclude, write
\begin{equation}\label{uni3}
\tilde x(z)=\tilde \wp (G) 
\qquad \textrm{and }\qquad
\tilde \wp'(u)\cdot \tilde y(z) =\tilde \wp '(G).
\end{equation}  
Because $\tilde x$ and $\tilde y$ are functions of $z$, the functions $\tilde \wp (G(u))$ and $\tilde \wp' (G(u))$ must be invariant under the transformation $u\mapsto u+2\pi i$. This, together with the particular form of the function $\tilde \wp$,  implies that the constant $\beta$ is a rational integer. We conclude that $\left. \tilde A_{xy}\right|_{z=1}$ is a rational integer, hence $\left. 2\alpha_{xy}\right|_{z=1,\delta=-2}$, whenever defined, is a rational integer.

Finally, in Section \ref{obtintegers}, we show that the quantities $\left. \alpha_{xy}\right|_{z=1,\delta=-2}$ are defined for some particular rational solutions over $\C (\delta ,z)$. The pair $(x,y)=(z,1)$ is a solution of Equation $\eqref{MD}$, which, seen as an equation over the field $\C (\delta )$, defines an elliptic curve. For $n\in \Z$ write $(x_n,y_n)=n(z,1)$, where addition is meant on the elliptic curve. We then show that for $n$ odd and $(x,y)=(x_n,y_n)$ the quantity $\left. \alpha_{xy}\right|_{z=1,\delta=-2}$ is defined and we have
$$
\left. \alpha_{x_ny_n}\right|_{z=1,\delta =-2}=\left. A_{x_ny_n}\right|_{z=1,\delta =-2}=n.
$$ 

Theorem \ref{mainLogic} will follow by setting $z_1=z-1$ and $z_2=\delta +2$. Indeed, the above allows, by the use of the predicate $\eval$, to define in a positive existential way the set of integers --- see Section \ref{Logic}. The undecidability results of Theorem \ref{mainLogic} follow by known  methods.

\section{The functions $x_zy^{-1}$ and $\frac{x-1}{(z-1)y}$}\label{lemmas}

\subsection{Analyticity of $x_zy^{-1}$}\label{analyticA}

In this section we show that for any solution $(x,y)$ of Equation \eqref{MD} over $\M$, the function $\frac{x_z}{y}$ is analytic, i.e. lies in $\H$. We start by recalling a well-known fact. 

\begin{theorem}\label{ufd} If $k$ is a field, the formal power series ring $K=k[[X_1,\dots,X_m]]$ is a unique factorization domain. If $k$ is a subfield of $\C$, then the subfield $k\{X_1,\dots,X_m\}$ of $K$, consisting  of the elements of $K$ which are germs of functions analytic at $(0,\dots,0)\in \C^m$, is a unique factorization domain.
\end{theorem}

From now on, if $g\in\M$ and $\rho$ is an irreducible (as a germ around a fixed point), we write $\ord_\rho(g)$
for the order of $g$ at $\rho$ (so $g$ is also considered as a quotient of germs around that point). 

\begin{lemma}\label{Alla}
Let $g\in\M$ and $\rho$ be an irreducible element of $\C\{z,\delta\}$.
\begin{enumerate}
\item If $\ord_\rho(g)\ne0$ and $\rho$ does not divide $\rho_z$, then $\ord_\rho(g_z)=\ord_\rho(g)-1$.
\item If $\ord_\rho(g)\ne0$ and $\rho$ divides $\rho_z$, then $\ord_\rho(g_z)\geq\ord_\rho(g)$.
\item If $\ord_\rho(g)=0$ then $\ord_\rho(g_z)\geq0$.
\end{enumerate}
\end{lemma}
\begin{proof}
Write $k=\ord_\rho(g)$ and suppose first that $k$ is not $0$. Write 
$$
g=\rho^k\frac{u}{v},
$$ 
where $u,v\in\H$ are chosen relatively prime (so $\rho$ divides neither $u$ nor $v$). We have 
$$
g_z=k\rho_z\rho^{k-1}\frac{u}{v}+\rho^k\frac{u_zv-uv_z}{v^2}
=\rho^{k-1}\frac{k\rho_zuv+\rho(u_zv-uv_z)}{v^2}
$$
and since $\rho$ does not divide $v$, we have 
$$
\ord_\rho(g_z)=k-1+\ord_\rho(k\rho_zuv+\rho(u_zv-uv_z)).
$$
If $\rho$ does not divide $\rho_z$, then it does not divide $\rho_zuv$, so 
$$
\ord_\rho(k\rho_zuv+\rho(u_zv-uv_z))=0.
$$ 
If $\rho$ divides $\rho_z$, then
$$
\ord_\rho(k\rho_zuv+\rho(u_zv-uv_z))
=1+\ord_\rho(k\rho_zuv\rho^{-1}+u_zv-uv_z)\geq1,
$$ 
hence $\ord_\rho(g_z)\geq k-1+1=k$.

Note that the case $k=0$ is trivial.
\end{proof}

The next lemma clarifies the condition for $\rho$ to divide or not $\rho_z$. 

\begin{lemma}\label{yoyo}
Let $\rho$ be an element of $\C\{z,\delta\}$. If $\rho$ is irreducible, then either $\rho$ does not divide $\rho_z$, or $\rho=\delta e$ for some unit $e$. 
\end{lemma}
\begin{proof}
It is obvious if $\rho$ is a polinomial. Otherwise, either it is $z$-regular (in its development in power series, there is a monomial not involving $\delta$), or $\rho=\delta e$ for some unit $e$ (because $\rho$ is irreducible). If it is $z$-regular, by the Weierstrass Preparation Theorem, there exists a unit $e$ and a Weierstrass polynomial $\omega$ in the variable $z$, such that $\rho=e\omega$. Therefore, if $e\omega=\rho$ divides $\rho_z=e_z\omega+e\omega_z$, then $\omega$ divides $e\omega_z$, hence $\omega$ divides $\omega_z$ (since $e$ is a unit). This is impossible since $\omega$ is a monic polynomial in $z$. 
\end{proof}

\begin{lemma}\label{xavier1}
If $(x,y)$ is a solution of Equation \eqref{MD} over $\M$, with $y\ne0$, 
then $x_zy^{-1}\in\H$.
\end{lemma}
\begin{proof}
Consider a solution $(x,y)$ of the Manin-Denef Equation
$$
(z^3+\delta z^2+z)y^2=x^3+\delta x^2+x
$$
over $\M$. Let $(\rho_1,\rho_2)$ be an arbitrary point of $B$ and $\rho$ be an irreducible at $(\rho_1,\rho_2)$. 
Write 
$$
\ord_\rho(x)=k,\quad \ord_\rho(x_z)=k', \quad\ord_\rho(y)=j\quad\textrm{and}\quad \ord_\rho(x^3+\delta x^2+x)=m.
$$ 
We will prove that $j$ is at most $k'$.
 
\emph{CASE 1:}  Assume that $k\ne 0$ and $\rho$ does not divide $f(\delta,z)=z^3+\delta z^2+z$. 

If $k>0$, then $k=m=2j>0$, hence we have 
$$
k'-j\ge k-1-j=2j-1-j=j-1\geq0
$$
by Lemma \ref{Alla}. 

If $k<0$, then we have $3k=m=2j<0$, so $j$ is a multiple of $3$. Hence we have 
$$
k'-j\ge k-1-j=\frac{2j}{3}-1-j=-\frac{1}{3}j-1\geq0
$$
by Lemma \ref{Alla}. 

\emph{CASE 2:} Assume that $k\ne 0 $ and $\rho$ divides $f(\delta,z)=z^3+\delta z^2+z$. Modulo a unit, either $\rho=z$, or $\rho=z^2+\delta z+1$ (in particular, $\rho$ does not divide $\rho_z$ by Lemma \ref{yoyo}). If $k>0$, then $k=2j+1$ and $j\geq0$, hence 
$$
k'-j=k-1-j=2j+1-1-j=j\geq0.
$$ 
If $k<0$, then we have $j<0$ and $3k=2j+1$. In particular, since $2j+1$ is a multiple of $3$, we have $j\leq-2$, hence 
$$
k'-j=k-1-j=\frac{2j+1}{3}-1-j=-\frac{1}{3}j-\frac{2}{3}\geq0.
$$ 

\emph{CASE 3:} Assume $k=0$. By Lemma \ref{Alla} we have $k'=\ord_\rho(x_z)\geq0$. Consider the equation:
\begin{equation}\label{mariac}
[(3z^2+2\delta z+1)y+2(z^3+\delta z^2+z)y_z]y=x_z(3x^2+2\delta x+1)
\end{equation}
obtained by differentiating both members of Equation \eqref{MD} with respect to $z$. 
We will write 
$$
\ord_\rho(3x^2+2\delta x+1)=\ell.
$$ 
Observe that since $k=0$ we have $m\geq0$, as otherwise we would have $\ord_\rho(x^2+\delta x +1)<0$, hence $\ord_\rho(x^2+\delta x)<0$, hence $\ord_\rho(x+\delta)<0$, hence $k=\ord_\rho(x)<0$. Similarly we have $\ell\geq0$.

\emph{Case 3a:} Assume $m=0$. If $\rho$ divides $f$ then we get $0=m=2j+1$ from Equation \eqref{MD}, which is impossible, and if $\rho$ does not divide $f$ then Equation \eqref{MD} gives $0=m=2j$, hence $j=0$, which implies $k'-j=k'\geq0$.

\emph{Case 3b\,:} Assume $m>0$. Since $\ord_\rho(x)=0$, we have
$$
\ord_\rho(x^2+\delta x+1)=\ord_\rho(x^3+\delta x^2+x)=m>0.
$$

\emph{Case 3b1:} Assume $\ell>0$. We will prove that $m=1$ and $k'\geq j=0$. We have 
$$
\ord_\rho(2x+\delta)=\ord_\rho(2x^2+\delta x)=
\ord_\rho((3x^2+2\delta x+1)-(x^2+\delta x+1))>0
$$
(since $\ell$ and $m$ are positive), hence $2x+\delta=\rho^nu$ for some positive integer $n$ and some $u$ such that $\ord_\rho(u)=0$. Therefore, we have
$$
\begin{aligned}
0&<\ord_\rho(x^2+\delta x+1)\\
&=\ord_\rho\left(\left(\frac{\rho^nu-\delta}{2}\right)^2+\delta\frac{\rho^nu-\delta}{2}+1\right)\\
&=\ord_\rho\left(\frac{\rho^{2n}u^2}{4}+\frac{\delta^2}{4}-\frac{\delta^2}{2}+1\right)\\
&=\ord_\rho\left(\frac{\delta^2}{4}-\frac{\delta^2}{2}+1\right)\\
&=\ord_\rho(\delta^2-4).
\end{aligned}
$$
Therefore, we have $\rho=(\delta\pm 2)e$ for some unit $e$. Suppose for example that $\rho=\delta+2$ (the general case is done similarly). We have
$$
0<m=\ord_{\delta +2}(x^{2}+\delta x+1)=
\ord_{\delta +2}((x-1)^{2}+(\delta +2)x),
$$ 
hence $\ord_{\delta +2}(x-1)>0$, and since 
$\ord_{\delta +2}((\delta +2)x)=1$ we deduce 
$$
\ord_{\delta+2}((x-1)^{2}+(\delta +2)x)=1.
$$ 
Hence we have $m=1$, and $\rho$ must divide $f$ (otherwise $m$ would be even by Equation \eqref{MD}). We conclude from Equation \eqref{MD} that $j=0$. We have $k'\geq0=j$ by Lemma \ref{Alla}.

\emph{Case 3b2:} Assume $\ell=0$. We have
$$
\ord_\rho(f_z)=\ord_\rho(x_z(3x^2+2\delta x+1))=k'+\ell=k'.
$$

Suppose $\rho$ does not divide $\rho_z$. Since $m=\ord_\rho(f)\ne0$, by Lemma \ref{Alla} we have 
$$
k'=\ord_\rho(f_z)=m-1.
$$
If $\rho$ does not divide $f$, then from Equation \eqref{MD} we have $m=2j>0$, hence $k'=m-1=2j-1\geq0$, hence $j\geq1$ and $k'\geq j$. If $\rho$ divides $f$ then from Equation \eqref{MD} we have $m=2j+1>0$, hence $k'=m-1=2j\geq0$, hence $k'\geq j$.

Now assume that $\rho$ divides $\rho_z$. Since $m=\ord_\rho(f)\ne0$, we have 
$$
k'=\ord_\rho(f_z)\geq m.
$$
Since $\rho$ divides $\rho_z$, we have $\rho\ne z$ and $\rho\ne z^2+\delta z+1$ by Lemma \ref{yoyo}, hence $\ord_\rho(f)=0$, and from Equation \eqref{MD} we obtain $m=2j>0$. Therefore, we have $k'\geq m=2j>0$, hence $k'\geq j$.
\end{proof}

\subsection{Behaviour around the point $(z,\delta)=(1,-2)$}

In this subsection, by \emph{order of a function $w$ of $\H$ at an irreducible $\rho$}, denoted by $\ord_\rho(h)$, we will mean the highest power of $\rho$ that divides that function in $\H$. We extend this notion to meromorphic functions in the usual way, and use the usual terminology of zeros and poles. 

\begin{lemma}\label{alphaA}
Whenever $\left. \alpha_{xy}\right|_{z=1,\delta=-2}$ is defined, $z-1$ is either a zero or a pole of $x-1$, and in this case we have
\begin{equation}\label{eqlemalphaA1}
\left. \alpha_{xy}\right|_{z=1,\delta=-2}\cdot \ord _{z-1} (x-1)= 
\left. A_{xy}\right|_{z=1,\delta=-2}.
\end{equation}
Moreover, if $\ord_{z-1}(x-1)<-2$ or if $\ord_{z-1}(x-1)>1$, then we have
$$
\left. \alpha_{xy}\right|_{z=1,\delta=-2}=0= 
\left. A_{xy}\right|_{z=1,\delta=-2}.
$$
\end{lemma}
\begin{proof}
Assume that $z-1$ is neither a zero nor a pole of $x-1$. By Equation \eqref{MD}, since it is not a pole of $x-1$, it is not a pole of $y$ either, hence $\alpha_{xy}$ has a pole at $z=1$ and therefore $\left. \alpha_{xy}\right|_{z=1}$ is undefined (hence $\left. \alpha_{xy}\right|_{z=1,\delta=-2}$ as well). 

Assume that $x-1$ has a zero or a pole at $z=1$. We have 
$$
\alpha_{xy}=\frac{x-1}{(z-1)y}
=\frac{x_z}{y} \left(x_z\frac{z-1}{x-1}\right)^{-1}
=
A_{xy}\left(x_z\frac{z-1}{x-1}\right)^{-1}.
$$
By considering the functions $x$ and $y$ as power series  in  $z-1$, the function   
$$
\frac{x_z}{x-1}
$$
is the logarithmic derivative of  $x-1$, so the value of 
$$
(z-1)\frac{x_z}{x-1}
$$ 
at $z=1$ is equal to the order of $x-1$ at $z-1$, which proves Equation \eqref{eqlemalphaA1}. 

Consider the case that $\ord_{z-1} (x-1)<-2$. By Equation \eqref{MD} we have $3\ord_{z-1}(x-1)=2\ord_{z-1}(y)$, hence $\ord_{z-1}(y)<-3$ and the quantity
$$
\ord_{z-1}(A_{xy}) =\ord_{z-1}\left(\frac{x_z}{y}\right)=\ord_{z-1}(x-1)-1-\ord_{z-1}(y)=\frac{-1}{3}\ord_{z-1}(y)-1
$$ 
is $>0$, hence $\left. A_{xy} \right|_{z=1,\delta=-2}=0$ and the conclusion of the Lemma holds by Equation \eqref{eqlemalphaA1}. 

If $\ord_{z-1}(x-1)>1$, then by Equation \eqref{MD} we have $\ord_{z-1}(y)=0$, and the quantity
$$
\ord_{z-1}(A_{xy}) =\ord_{z-1}\left(\frac{x_z}{y}\right)=\ord_{z-1}(x-1)-1
$$ 
is $>0$, hence $\left. A_{xy} \right|_{z=1,\delta=-2}=0$ and the conclusion of the Lemma holds by Equation \eqref{eqlemalphaA1}. 
\end{proof}

\section{Integrality of $x_zy^{-1}$ at a singular point}\label{secuni}

\subsection{Parametrisation}\label{secparam}

Consider the affine curve $\tilde E$ (in coordinates $(X,Y)$) defined by the equation
$$
Y^2=X(X-1)^2.
$$
We will consider the following notation till the end of this work (inspired by the analogy with elliptic curves). 

\begin{notation}
\begin{itemize}
\item If $g$ is a function of the variable $u$ only, we will use the usual notation for the first and second derivatives of $g$: $g'$ and $g''$.
\item $\tilde f(z)=z(z-1)^2$
\item $\displaystyle{\twp (u) =\left(\frac{1+e^u}{1-e^u}\right)^2}$
\item $\displaystyle{\tilde \xi (u) =2\frac{1+e^u}{1-e^u}}$
\item $\tilde W=(\twp,\twp')$
\item $\tilde \omega =2\pi i$
\item $\tilde \Lambda =\{k\tilde \omega\colon k\in \Z\}$
\end{itemize}
\end{notation}

The following Lemma lists some properties of these functions. 

\begin{lemma}\label{properties}
The functions $\twp $, $\twp '$, and $\tilde \xi$ satisfy the following:
\begin{enumerate}
\item (Functional equations) We have
$$
(\twp')^2 =\twp(\twp-1)^2=\tilde f\circ\twp,
$$
hence also
$$
\twp''=\frac{1}{2}\tilde f_z\circ\twp.
$$
\item (Parametrisation) The pair $\tilde W =(\twp ,\twp')$ can be seen as an onto map 
$$
\C\to(\tilde E\cup \{ \infty\})\setminus \{ (1,0)\}.
$$  
\item\label{propertiesper} (Periodicity) We have
$$
\twp (u+\tilde \omega ) =\twp (u),\quad
 \twp' (u+\tilde \omega ) =\twp' (u),\quad
 \tilde \xi (u+\tilde \omega)= \tilde \xi (u),
$$
and if $\tilde W(u_1)=\tilde W(u_2)$, then $u_1-u_2\in\tilde\Lambda$.
\item We have
$$
\tilde \xi' = \twp-1,\quad
 \twp =\frac{1}{4}\tilde \xi ^2,
\quad\textrm{and}\quad
\twp' = \frac{1}{2} (\twp -1)\tilde \xi.
$$
\end{enumerate}
\end{lemma}
\begin{proof}
We only comment on the last part of item 3, as the rest of the proof is straightforward computation. From $\twp(u_1)=\twp(u_2)$, we deduce that either $\xi(u_1)=\xi(u_2)$, in which case we easily get $e^{u_1}=e^{u_2}$, or $\xi(u_1)=-\xi(u_2)$, in which case
we get $e^{u_1+u_2}=-1$. The second case is impossible, because the last relation of item 4 implies $\twp '(u_1)=-\twp'(u_2) $, which contradicts the hypothesis.
\end{proof}

\subsection{The Group Law}

The affine equation 
\begin{equation}\label{E}
Y^2=X^3+\delta X^2+X
\end{equation}
defines an affine surface in the triple of variables $(\delta,X,Y)$. We consider it as a bundle of affine curves $E_\delta$, parametrized by the parameter $\delta$, and we write $\E_\delta$ for each corresponding projective curve. For each fixed $\delta$, the point $[W:X:Y]=[0:0:1]$ of the projective equation 
$$
WY^2=X^3+\delta WX^2+W^2X
$$ 
is a point of $\E_\delta$, called \emph{point at infinity} and denoted by $\infty$. For $\delta \ne \pm 2$, the polynomial $f(\delta,z)=z^3+\delta z^2+z$ in the variable $z$ is non-singular (i.e. has no multiple zeros), hence $\E_\delta$ is an elliptic curve with neutral $\infty$. We now recall the addition law on each $\E_\delta$, which we denote by $\oplus$, in affine coordinates:
 \begin{enumerate}
\item $\infty$ is the neutral element.
\item If $P_1=(a,b)$ and $P_2=(a,-b)$, then $P_1\oplus P_2=\infty$.
\item If $P_1=(a_1,b_1)$ and $P_2=(a_2,b_2)$, with $a_1\ne 0$ and $a_2=0$, then $P_1\oplus P_2=(a,b)$, where $a=a_1^{-1}$ and $b=-b_1a_1^{-2}$.
\item If $P_1=(a_1,b_1)$ and $P_2=(a_2,b_2)$, with $a_1a_2\ne 0$ so that case 2 does not apply, then $P_1\oplus P_2=(a,b)$, with
$$
a=\frac{(b_1-a_1m)^2}{a_1a_2}
\quad\textrm{and}\quad
b=-b_1-m(a-a_1),
$$
where
$$
m=
\begin{cases}
\displaystyle{\frac{b_2-b_1}{a_2-a_1}}&\textrm{if $a_1\ne a_2$,}\vspace{10pt}\\
\displaystyle{\frac{3a_1^2+2\delta a_1+1}{2b_1}}&\textrm{if $a_1=a_2$ and $b_1\ne0$.}
\end{cases}
$$
\end{enumerate}

Observe that $\ominus (a,b)=(a,-b)$ and the points of the form $(a_0,0)$
are exactly the points of order $2$ --- so there are three such points.
 
For $\delta =\pm 2$, Equation \eqref{E} defines a curve of genus $0$. Nevertheless its points, with the exception of the point $(1,0)$, form a group under the same law, with the same neutral element. Observe that the point $(1,0)$ cannot be included in the group structure of $\tilde E=E_{-2}$, as we would have $P\oplus (1,0)=(1,0)$ for any point $P$ of $\tilde E$. Therefore, the only point of order $2$ on $\tilde E$ is $(0,0)$. 
 
\begin{lemma}\label{lawtransfer}
The addition on the complex plane transfers through the pair $(\twp,\twp')$. This makes the set $\tilde E\setminus\{(1,0)\}$ into a group, under the law which is given by the same formulas giving $\oplus$. 
\end{lemma} 
\begin{proof}
We leave to the reader verifying that the following 
$$
(\twp (u_1+u_2), \twp'(u_1+u_2))=(\twp (u_1),\twp'(u_1))\oplus (\twp(u_2),\twp'(u_2))
$$
holds in all cases, as this is a straightforward computation.
\end{proof}

\subsection{A representation lemma}\label{secrep}

In this section, we fix $\rho\ge 1$ a real number or $+\infty$. Given an arbitrary $H\in\H_z(\bar D(0,\rho))$, we will solve the following differential equation
$$
H=
 \beta +\gamma(z-1)+\tilde f \tilde h_z+\frac{1}{2}\tilde f_z \tilde h
$$
in the unknowns $\beta$, $\gamma$ and $h$, where $\beta,\gamma\in\C$ and $h\in\H_z(\bar D(0,\rho))$ (indeed we will show that $\beta$, $\gamma$ and $h$ exist and are uniquely determined by $H$). 

\begin{lemma}\label{central0}
Let $b\in \H_{z} (\bar D(0,\rho))$. 
\begin{enumerate}
\item There is a unique function $g\in \H_{z} (\bar D(0,\rho))$ such that
\begin{equation}\label{diff1}
zg_z+\frac{1}{2}g=b.
\end{equation}
\item There is a unique function $g\in \H_{z} (\bar D(0,\rho))$ and a unique $\gamma\in \C$ such that 
\begin{equation}\label{diff2}
z(z-1)g_z+\frac{1}{2}(3z-1)g=b-\gamma.
\end{equation}
\end{enumerate}
\end{lemma}
\begin{proof}
We prove item 1. We will find a solution $g$ in power series around the point $z=0$ and then we will observe that its radius of convergence is equal to the radius of convergence of a power series around $z=0$ for the function $b$. Write $b=\beta _0+\dots +\beta_n z^n+\cdots$. If 
$$
g=\gamma _0+\dots +\gamma_n z^n+\cdots,
$$
from Equation \eqref{diff1} we obtain, for any $n\geq 0$, 
$$
\left(n+\frac{1}{2}\right) \gamma _n=\beta _n.
$$
Define the $\gamma_n$ by the last relation. Since the quantity $n+\frac{1}{2}$ never vanishes, this relation defines the sequence $(\gamma_n)$ uniquely (hence also $g$, if it exists). Since 
$$
\lim_{n\rightarrow \infty} \left(n+\frac{1}{2}\right)^{\frac{1}{n}}=1,
$$ 
the radii of convergence of the power series for $b$ and for $g$ are the same, which proves the existence of $g$ with the required properties. 

From item 1, there exists a solution $h\in \H_z(\bar D(0,\rho))$ of Equation \eqref{diff1}. We obtain a function $g$ and a constant $\gamma$, as required in Equation \eqref{diff2}, by setting $h=(z-1)g+2\gamma$. Indeed, we have $h_z=g+(z-1)g_z$, hence 
$$
\begin{aligned}
b&=zh_z+\frac{1}{2}h\\
&=z(g+(z-1)g_z)+\frac{1}{2}((z-1)g+2\gamma)\\
&=z(z-1)g_z+\left(z+\frac{1}{2}z-\frac{1}{2}\right)g+\gamma.
\end{aligned}
$$
We now prove unicity. For any solution $(g,\gamma)$ of Equation \eqref{diff2}, 
$$
h=(z-1)g+2\gamma
$$ 
is a solution of Equation \eqref{diff1} by the above computation. So if $(g_1,\gamma_1)$ and $(g_2,\gamma_2)$ are solutions of Equation \eqref{diff2}, then we have 
$$
(z-1)g_1+2\gamma_1=(z-1)g_2+2\gamma_2
$$ 
by the unicity of the solution of Equation \eqref{diff1}. Hence, $(z-1)(g_1-g_2)$ is constant, so $g_1=g_2$, from which we deduce $\gamma_1=\gamma_2$. 
\end{proof}

\begin{remark}
The way to compute the constant $\gamma$ of item 2 in Lemma \ref{central0}, according to the proof we gave is: Compute $h$ from $b$ by finding the coefficients of $h$ around the point $z=0$ - apparently an infinite procedure - and then compute $\gamma =\frac{1}{2}h(1)$. In this sense our proof is not constructive. 
\end{remark}

\begin{notation}\label{notGH}
Given $\beta,\gamma\in \C$ and $h\in\H_z(\bar D(0,\rho))$, consider the function
$$
H_{\beta,\gamma,h}=\beta +\gamma (z-1)+\tilde f h_z+\frac{1}{2} \tilde f_zh
$$
in the variable $z$. 
\end{notation}

\begin{lemma}[Representation lemma]\label{central}
For any function $g\in \H_z(\bar D(0,\rho))$, there are unique constants $\beta$ and $\gamma$, and there is a unique function $h\in \H_z (\bar D(0,\rho))$, such that $g=H_{\beta,\gamma,h}$.
\end{lemma}
\begin{proof}
Let $\beta=g(1)$. First observe that the function
$$
\frac{g-\beta}{z-1}
$$
lies in $\H_z(\bar D(0,\rho))$. By item 2 of Lemma \ref{central0}, applied to this function, the equation 
$$
z(z-1)h_z+\frac{1}{2}(3z-1)h=\frac{g-\beta}{z-1}-\gamma
$$
has a unique solution $(h,\gamma)$ with $h\in\H_z(\bar D(0,\rho))$ and $\gamma\in\C$. We have then:
$$
z(z-1)^2h_z+\frac{1}{2}(z-1)(3z-1)h=g-\beta-\gamma(z-1)
$$
hence
$$
\begin{aligned}
g&=\beta+\gamma(z-1)+z(z-1)^2h_z+\frac{1}{2}(3z^2-4z+1)h\\
&=\beta+\gamma(z-1)+\tilde f h_z+\frac{1}{2}\tilde f_z h.
\end{aligned}
$$
Therefore, we have $g=H_{\beta,\gamma,h}$. 

We prove the uniqueness of $\beta$, $\gamma$ and $h$. By linearity, it suffices to prove that if $H_{\beta,\gamma,h}=0$ then all  $\beta$, $\gamma$ and $h$ are equal to zero. Evaluating at $z=1$, we have immediately $\beta=0$, hence
$$
\gamma(z-1)+\tilde f h_z+\frac{1}{2}\tilde f_z h=0.
$$
We get
$$
\gamma(z-1)+z(z-1)^2 h_z+\frac{1}{2}(3z^2-4z+1)h=0
$$
hence
$$
z(z-1)h_z+\frac{1}{2}(3z-1)h=-\gamma.
$$
By item 2 of Lemma \ref{central0} (taking $b=0$), we deduce that $h=0$ and $\gamma=0$ (since $(0,0)$ is a solution, and the solution is uniquely determined by $b$). 
\end{proof}

\subsection{Uniformisation}\label{secuni}

Assume that $\delta+2$ is not a pole or a zero of $y$. Note that by Equation \eqref{MD}, $\delta+2$ is not a pole of $x$, hence the pair 
$$
(\tilde x,\tilde y)=(\left. x\right|_{\delta=-2},\left. y\right|_{\delta=-2})
$$ 
is well defined and satisfies the functional equation
\begin{equation}\label{tMD}\tag{$\widetilde{MD}$}
\tilde f\cdot \tilde y^2=\tilde f \circ\tilde x.
\end{equation}
Since $\delta+2$ is not a zero of $y$, $\tilde y$ is not the zero function, hence the quotient $\tilde x_z \tilde y^{-1}$ is well defined, and therefore is (trivially) equal to $\tilde A_{xy}=\left. A_{xy}\right|_{\delta=-2}$. In particular, it lies in $\H_z (\C)$. By Lemma \ref{central}, there exist unique constants $\beta,\gamma\in\C$ and a unique $h\in\H_z (\C)$ such that $\tilde A_{xy}=H_{\beta,\gamma,h}$. So a solution $(x,y)$ of Equation \eqref{MD} determines uniquely $\beta$, $\gamma$ and $h$. 

\begin{notation}\label{notG0}
Given $\beta,\gamma\in \C$ and $h\in\H_z(\C)$, consider the function 
$$
G_{\beta,\gamma,h}=\beta u+\gamma \tilde \xi+\twp' \cdot h\circ\twp
$$ 
of the variable $u$. 
\end{notation}

Note that $G_{\beta,\gamma,h}$ is uniquely determined by the pair $(x,y)$. 

\begin{notation}\label{notG}
When $\delta+2$ is not a zero or a pole of $y$, we will write 
$$
G_{xy}=G_{\beta,\gamma,h}.
$$
\end{notation}

\begin{lemma}\label{lemG}
For any $\beta,\gamma\in \C$ and $h\in\H_z(\C)$, we have
$$
G_{\beta,\gamma,h}'=H_{\beta,\gamma,h}\circ\twp.
$$
\end{lemma}
\begin{proof}
By Lemma \ref{properties}, we have:
$$
\begin{aligned}
G_{\beta,\gamma,h}'&=\beta +\gamma (\twp -1)+(\twp ')^2\cdot h_z\circ\twp+\twp''\cdot h \circ\twp\\
&=\beta +\gamma (\twp -1)+\tilde f\circ \twp\cdot  h_z\circ\twp+\frac{1}{2}\tilde f_z\circ \twp\cdot h\circ\twp.
\end{aligned}
$$
\end{proof}

Note that, by Lemma \ref{lemG}, we have
$$
G_{xy}'=G_{\beta,\gamma,h}'=H_{\beta,\gamma,h}\circ\twp=\tilde A_{xy}\circ\twp.
$$

Consider an element $\tilde s$, algebraic over $\C(z)$, satisfying $\tilde s^2=z(z-1)^2$, so that $(z,\tilde s)$ lies on the curve $\tilde E$. The pair $(\tilde x,\tilde s\tilde y)$ is a pair of non-constant functions, because if $\tilde x\in\C$, then $\tilde f \tilde y^2\in\C$, hence, since we assumed that $\tilde y$ is not the zero function, $\tilde f$ would be a square in $\M_z (\C)$, which is not the case. So the pair $(\tilde x,\tilde s\tilde y)$ may be seen as a map
$$
\begin{array}{ccc}
\tilde E&\longrightarrow&\tilde E\cup\{\infty\}\\
(z,\tilde s)&\longmapsto&(\tilde x(z),\tilde s\tilde y(z)).
\end{array}
$$
 
In the next lemma, we will prove that the following diagram
$$
\begin{CD}
\tilde E     @>(\tilde x,\tilde y)>>  \tilde E\cup\{\infty\}\\
@A(\twp,\twp')AA							     @AA(\twp,\twp')A\\
\C\setminus\tilde\Lambda @ >G_{xy}>>  \C\setminus\tilde\Lambda\\
\end{CD}
$$
is \emph{almost} commutative (in the sense of the next lemma). 

\begin{lemma}[Uniformisation Lemma]\label{uniform}
Assume that $\delta+2$ is not a zero or a pole of $y$. There exists $\mu\in\{0,\pi i\}$ such that for any $u\in \C\setminus\tilde\Lambda$ we have
$$
\tilde x \circ\twp (u)=\twp (G_{xy} (u)+\mu)\quad\textrm{and}\quad
\twp '(u)\cdot \tilde y \circ\twp (u)=\twp' (G_{xy} (u)+\mu).
$$
\end{lemma}
\begin{proof}
Write $\hat x =\tilde x \circ\twp$ and $\hat y=\tilde y \circ\twp$. Since $\tilde x$ and $\tilde y$ are non-constant, there is a point $u_0\in \C\setminus\tilde\Lambda$ such that 
$$
(\hat x (u_0),\twp '(u_0) \hat y (u_0)),
$$
is finite and such that the Jacobian matrix of the function $(\twp , \twp ' )$  does not vanish at $u_0$. By the Implicit Function Theorem for analytic functions there is a function $G_0$ of the variable $u$, analytic on an open neighborhood of $u_0$, such that for each $u$ in that neighborhood we have 
$$
\hat x =\twp \circ G_0\qquad\textrm{and} \qquad \twp ' \hat y =\twp '\circ G_0.
$$ 
Since
$$
\begin{aligned}
G_{xy}'&=\tilde A_{xy}\circ \twp =\frac{\tilde x_z}{\tilde y}\circ \twp =
\frac{\hat x'}{\twp'  \hat y}
=\frac{G_0'\cdot \twp'\circ G_0}{\twp'\circ G_0}=G_0',
\end{aligned}
$$
the function $G_{xy}-G_0$ is equal to a constant $\mu$, so that we have
$$
\hat x=\twp \circ(G_{xy} -\mu)\quad\textrm{and}\quad
\twp'  \hat y =\twp ' \circ(G_{xy} -\mu ).
$$    
By Lemma \ref{lawtransfer}, we have
$$
(\hat x,\twp'  \hat y)=(\twp \circ(G_{xy} -\mu ),\twp ' \circ(G_{xy} -\mu ))=
(\twp\circ G_{xy},\twp ' \circ G_{xy})\ominus(\twp(\mu),\twp'(\mu)),
$$
so that the function
$$
(\hat x,\twp'  \hat y)\ominus (\twp \circ G_{xy},\twp ' \circ G_{xy})
$$
is constant. This is true on the domain of definition of the function $G_0$, so by the local-to-global property of meromorphic functions, this holds for any $u$ in $\C\setminus\tilde\Lambda$. 

 We claim that this constant must be a point of order $2$ on $\tilde E$. Indeed, observe that $G_{xy}$ is an odd function of $u$, because $\tilde\xi$ and $\twp'$ are odd and $\twp$ is even. Therefore, the transformation $u\mapsto -u$ maps the point $(\twp\circ G_{xy}(u),\twp' \circ G_{xy}(u))$ to its negative on $\tilde E$. The same happens with $(\hat x,\twp'  \hat y)=(\tilde x \circ\twp,\twp' \cdot  \tilde y \circ\twp)$. So, the difference $(\hat x,\twp'  \hat y)\ominus (\twp \circ G_{xy},\twp ' \circ G_{xy})$ is an odd function with respect to the group law of $\tilde E$, and since it is constant, it has to be a point of order at most $2$ on $\tilde E$. Therefore, one can take $\mu\in\{0,\pi i\}$, because the only point of order $2$ on $\tilde E$ is $(0,0)=(\twp(\pi i),\twp'(\pi i))$.
\end{proof}

\subsection{Integral values}\label{intval}

The Uniformisation Lemma \ref{uniform} gives us a first access to the integers. 

\begin{lemma}\label{intvallem}
Assume that $\delta+2$ is not a pole or a zero of $y$. The quantity $\left. A_{xy}\right|_{z=1,\delta=-2}=\left. A_{xy}\right|_{\delta=-2,z=1}$ is a rational integer.
\end{lemma}
\begin{proof}
Because $\twp$ and $\twp'$ are periodic of period $\tilde\Lambda$ (see Lemma \ref{properties}), and thanks to the Uniformisation Lemma \ref{uniform}, there exists $\mu\in\{0,2\pi i\}$ such that, for any $\lambda\in\tilde\Lambda$ and for any $u\in\C\setminus\tilde\Lambda$, we have
$$
\twp (G_{xy} (u+\lambda)+\mu)=\tilde x \circ\twp (u+\lambda)=
\tilde x \circ\twp (u)=\twp (G_{xy}(u)+\mu)
$$
and
$$
\twp' (G_{xy}(u+\lambda)+\mu)=\twp'(u+\lambda)\cdot \tilde y \circ\twp (u+\lambda)=
\twp'(u)\cdot\tilde y \circ\twp (u)=\twp' (G_{xy}(u)+\mu).
$$
Hence $G_{xy}(u+\lambda)+\mu$ is congruent to $G_{xy}(u)+\mu$ modulo $\tilde\Lambda$ (by Lemma \ref{properties}, item \ref{propertiesper}), so $G_{xy}(u+\lambda)$ is congruent to $G_{xy}(u)$ modulo $\tilde\Lambda$.

Let $\beta$, $\gamma$ and $h$ be such that $G_{xy}=G_{\beta,\gamma,h}$. Since 
$$
G_{xy}(u)=\beta u+\gamma\tilde\xi(u)+\twp(u)\cdot h\circ\twp(u),
$$ 
and $\twp$, $\twp'$ and $\tilde\xi$ are periodic of period $\tilde\Lambda$, we have 
$$
\beta\lambda=\beta (u+\lambda)-\beta u=G_{xy}(u+\lambda)-G_{xy}(u)\in\tilde\Lambda.
$$
Hence $\beta$ is an integer. 
\end{proof}

\section{Integrality of $\frac{x-1}{(z-1)y}$ at a singular point}\label{obtintegers}

We can now obtain the integers from the function $\alpha_{xy}$ in a diophantine way, which is what we will need for the logical conclusions. Indeed, we will show that the odd integers are values of  $\alpha_{xy}$ at $(z,\delta )=(1,-2)$.

Given $\delta\in\C$, let $s$ be a square root of $f(\delta,z)$ in an algebraic closure of $\M$. Consider $(z,s)$ as a point on $E_\delta$. It is well known that for each $n\in\Z$, we have 
$$
n(z,s)=\overbrace{(z,s)\oplus\dots\oplus(z,s)}^{n\textrm{ times}}=(x_n(z,\delta),sy_n(z,\delta))
$$ 
for some rational functions $x_n$ and $y_n$ over $\Q$. We now prove some special properties of the maps $x_n$ and $y_n$. 

\begin{lemma}\label{wellknown}
For any $\delta\notin\{-2,2\}$ and for any $n\in\Z\setminus\{0\}$, we have 
$$
\frac{\partial x_n}{\partial z}=ny_n.
$$ 
\end{lemma}
\begin{proof}
See, for example, \cite[Lemma 3.1]{Vidaux03}.
\end{proof}

\begin{lemma}\label{groupLaw} 
For $\delta\notin\{-2,2\}$ and for any $n\in \Z\setminus \{ 0\}$, we have: 
\begin{enumerate}
\item 
$$
\ord_{z-1}(x_n-1)=
\begin{cases}
1&\textrm{if $n$ is odd,}\\
-2&\textrm{if $n\in4\Z$,}\\
0&\textrm{if $n\in4\Z+2$}.
\end{cases}
$$
\item 
$$
\left. \alpha_{x_ny_n}\right|_{z=1}=
\begin{cases}
n&\textrm{if $n$ is odd,}\\
-\frac{n}{2}&\textrm{if $n\in4\Z$,}\\
0&\textrm{if $n\in 4\Z+2$}.
\end{cases}
$$ 
\end{enumerate}
\end{lemma}
\begin{proof}
\begin{enumerate}
\item From the addition law, we have 
$$
x_2=\frac{(z^2-1)^2}{4f(\delta,z)}.
$$ 
Therefore, $(0,0)$ is a point of order $2$ on $E_\delta$ and $2(1,f(\delta,1))=(0,0)$, so $P_0=(1,f(\delta,1))$ is a point of order $4$. Hence, for an arbitrary $n\in\Z\setminus\{0\}$, we have the cases: 
\begin{itemize}
\item[(a)] $n\cdot P_0=\pm P_0$ for $n$ odd.
\item[(b)] $n\cdot P_0=(0,0)$ for $n\equiv 2\mbox{ mod }4$. 
\item[(c)] $n\cdot P_0=\infty$ for $n\equiv 0\mbox{ mod }4$. 
\end{itemize}
Note that the first coordinate of $nP_0$ is $x_n(1)$, and that we have
\begin{equation}\label{eqoddevenmod4}
\left(\frac{1}{n}\frac{\partial x_n}{\partial z}\right)^2=f(\delta,x_n)
\end{equation} 
(from Equation \eqref{MD} and Lemma \ref{wellknown}). 

If $n$ is odd, then $z-1$ is a zero of $x_n-1$ by item (a), hence it is not a zero of $f(\delta,x_n)=x_n^3+\delta x_n^2+x_n$, so neither it is a zero of $\frac{\partial x_n}{\partial z}=\frac{\partial (x_n-1)}{\partial z}$ by Equation \eqref{eqoddevenmod4}. In particular, it is a zero of multiplicity $1$ of $x_n-1$. 

If $n\in4\Z$, then $z-1$ is a pole of $x_n$ by item (c). Write $\ell$ for the order of this pole. From Equation \eqref{eqoddevenmod4} we have $2(\ell-1)=3\ell$, hence $\ell=-2$.

If $n\in4\Z+2$, then $z-1$ is a zero of $x_n$ by item (b), so it is not a zero of $x_n-1$. 

\item Observe that 
$$
\frac{y_n}{x_n-1}=\frac{1}{n}\frac{\frac{\partial (x_n-1)}{\partial z}}{x_n-1}
$$
by Lemma \ref{wellknown}, hence 
$$
\left. \frac{1}{\alpha_{x_ny_n}}\right|_{z=1}=\left. \frac{(z-1)y_n}{x_n-1}\right|_{z=1}
=\left. \frac{1}{n}(z-1)\frac{\frac{\partial (x_n-1)}{\partial z}}{x_n-1}\right|_{z=1}
$$ 
is $1/n$ times the order at $z=1$ of $x_n-1$ (by general properties of logarithmic derivatives). 
\end{enumerate}
\end{proof}

The next three lemmas deal with the behaviour of $x_n$ and $y_n$ at $\delta+2$. 

\begin{lemma}\label{Liz}
For any non-zero integer $n$, we have $A_{x_ny_n}=n$. 
\end{lemma}
\begin{proof}
By Lemma \ref{wellknown}, $A_{x_ny_n}$ is the constant function $n$ whenever $\delta\ne\pm 2$. Since $A_{x_ny_n}$ is analytic everywhere by Lemma \ref{xavier1}, this is still true at $\delta=-2$. 
\end{proof}

\begin{lemma}\label{productformula}
Let $k,n\in\Z$ be such that $n\ge2$, and $1\le k<n$. We have
\begin{equation}\label{product}
x_{n+k} x_{n-k}=\frac{(x_kx_n-1)^2}{(x_n-x_k)^2}
\end{equation}
and
\begin{equation}\label{x2}
x_{2n}=\frac{(x_n-x_n^{-1})^2}{4(x_n+\delta +x_n^{-1})},
\end{equation}
and the same formulas hold true for $\tilde x_n$, $\tilde x_{n-k}$, $\tilde x_{n+k}$ and $\tilde x_{2n}$. 
\end{lemma}
\begin{proof}
We show that the first equation holds and leave the second one to the reader. By the Addition Law on the curve $E_\delta$, for any $k,n$ satisfying the hypothesis, we have (the fourth equality comes from Equation \eqref{MD}) --- here we write $f$ for $f(\delta,z)$: 
$$
\begin{aligned}
x_{n+k}\cdot x_{n-k}&=\frac{f\cdot (x_ny_k-x_ky_n)^2}{x_kx_n(x_n-x_k)^2}
\cdot \frac{f\cdot (x_ny_k+x_ky_n)^2}{x_kx_n(x_n-x_k)^2}\\
&=\frac{f^2\cdot (x_n^2y_k^2-x_k^2y_n^2)^2}{x_k^2x_n^2(x_n-x_k)^4}\\
&= \frac{x_k^2 x_n^2\cdot \left(\frac{f\cdot  y_k^2}{x_k^2}-\frac{f\cdot  y_n^2}{x_n^2}\right)^2}{(x_n-x_k)^4}\\
&=\frac{x_k^2\cdot x_n^2\cdot \left((x_k+\delta +\frac{1}{x_k})-(x_n+\delta +\frac{1}{x_n})\right)^2}{(x_n-x_k)^4}\\
&= \frac{x_k^2\cdot x_n^2\cdot \left(x_n-x_k-\left(\frac{1}{x_k}-\frac{1}{x_n}\right)\right)^2}{(x_n-x_k)^4}\\
&= \frac{x_k^2\cdot x_n^2\cdot \left(1-\frac{1}{x_k\cdot x_n}\right)^2}{(x_n-x_k)^2}\\
&= \frac{(x_kx_n-1)^2}{(x_n-x_k)^2}.
\end{aligned}
$$
The formula for $\tilde x_n$ is proven in exactly the same way (since the group law is given by the same formulas). 
\end{proof}

\begin{lemma}\label{quotienttilde}
For each $n\ge 1$, the quotient 
$$
e_n=\frac{x_n}{\tilde x_n}
$$ 
can be written as a power series in $\delta +2$ (so with coefficients in $\Q (z)$ and non-negative exponents) with constant term equal to $1$. 
\end{lemma}
\begin{proof}
We prove it by induction on $n$. Assume $n\ge 2$ (and observe that $x_{-n}=x_n$). By Equation \eqref{product} (taking $k=1$), we have
$$
\tilde x_{n+1}\tilde x_{n-1}e_{n+1} e_{n-1}=\left(\frac{z\tilde x_ne_n-1}{e_n\tilde x_n-z}\right)^2
=\left(\frac{ze_n-\tilde x_n^{-1}}{e_n-z\tilde x_n^{-1}}\right)^2
$$
hence, by Equation \eqref{product} for $\tilde x_{n+1}\tilde x_{n-1}$, 
$$
\left(\frac{z\tilde x_n-1}{\tilde x_n-z}\right)^2e_{n+1} e_{n-1}=\left(\frac{ze_n-\tilde x_n^{-1}}{e_n-z\tilde x_n^{-1}}\right)^2
$$
hence (doing a cross product)
$$
\left(\frac{e_n-z\tilde x_n^{-1}}{\tilde x_n-z}\right)^2e_{n+1} e_{n-1}=\left(\frac{ze_n-\tilde x_n^{-1}}{z\tilde x_n-1}\right)^2
$$
and finally
$$
\left(\frac{e_n-z\tilde x_n^{-1}}{1-z\tilde x_n^{-1}}\right)^2e_{n+1} e_{n-1}=\left(\frac{ze_n-\tilde x_n^{-1}}{z-\tilde x_n^{-1}}\right)^2.
$$
Since $1-z\tilde x_n^{-1}$ and $z-\tilde x_n^{-1}$ do not vanish as functions of $z$ (because for $n\ge 2$, the rational functions $\tilde x_n$ have degree $n$), if $e_n$ is a power series in $\delta+2$ with coefficients in $\Q(z)$ and with constant term $1$, then also $e_{n-1}e_{n+1}$ has this property. The claim follows by induction because $e_1=1$ and $e_2$ do have the property. Indeed, we have 
$$
\tilde x_2=\frac{(z+1)^2}{4z}
$$
and 
$$
x_{2}=\frac{(z^2-1)^2}{4\cdot (z^3+\delta z^2+z)}
$$
so 
$$
\begin{aligned}
\frac{x_2}{\tilde x_2}&=\frac{4z(z^2-1)^2}{4(z^3+\delta z^2+z)(z+1)^2}\\
&=\frac{(z-1)^2}{z^2+\delta z+1}\\
&=\frac{(z-1)^2}{1+(\delta +2)z-2z+z^2}\\
&=\frac{(z-1)^2}{(z-1)^2+z(\delta +2)}\\
&=\frac{1}{1+\frac{z}{(z-1)^2}(\delta +2)}.
\end{aligned}
$$
\end{proof}

\begin{lemma}\label{Poly}
For any integer $n\ne0$, the quantity $\delta+2$ is not a zero or a pole of $y_n$. 
\end{lemma}
\begin{proof}
If $\delta+2$ were a zero or a pole of $y_n$, then it would be a zero or a pole of $x_n$ as well (by Equation \eqref{MD}), but this would contradict Lemma \ref{quotienttilde} because $\tilde x_n$ is a non-constant function of $z$ only. 
\end{proof}

We are now able to prove our main theorem about the endomorphism maps $x_n(z,\delta)$ and $y_n(z,\delta)$. 

\begin{corollary}\label{Sofia}
For any odd integer $n$, $\left. \alpha_{x_ny_n}\right|_{z=1,\delta=-2}$ is defined, and we have
$$
\left. \alpha_{x_ny_n}\right|_{z=1,\delta =-2}=\left. A_{x_ny_n}\right|_{z=1,\delta =-2}=n.
$$ 
\end{corollary}
\begin{proof}
The first assertion follows immediately from Lemmas \ref{groupLaw} and \ref{Poly}. The first equality follows by combining Lemmas \ref{alphaA} and \ref{groupLaw}, whereas the second equality is Lemma \ref{Liz}. 
\end{proof}

\section{Logical consequences}\label{Logic}

We prove Theorem \ref{mainLogic}. From Lemma \ref{alphaA}, Lemma \ref{intvallem} and Corollary \ref{Sofia}, it follows: 
\begin{lemma}\label{integral}
As $(x,y)$, with $y\ne0$, ranges over the set of solutions of Equation \eqref{MD} over $\M=\M_{z,\delta}$, the set of finite values of $2\alpha_{xy}$ is contained in $\Z$ and contains $4\Z+2$.
\end{lemma}

Consider  the set $S$ defined as follows: $n\in S$ if and only if 
\begin{multline}\notag
n\in \C \wedge \exists a,b,x,y,v\in \M\\
(z+\delta z^2+z) b^2=a^3+\delta a^2+a\\
\wedge y\ne 0\wedge  
(x,y)=2(a,b)\oplus (z,1)\wedge 2(x-1)=(z-1) y v\wedge \eval (v-n).
\end{multline}
It follows from Lemma \ref{integral} that $S$ is contained in the set of rational integers and contains $4\Z+2$. Indeed, if $n=2(2k+1)\in 4\Z+2$, then we may choose $(a,b)=(x_k,y_k)$, $(x,y)=(x_{2k+1},y_{2k+1})$, and $v=2\alpha_{xy}$. Conversely, according to Lemma \ref{alphaA} combined with Lemma \ref{intvallem}, any $(x,y)$ as in the formula has the property $\left. 2\alpha_{xy}\right|_{z=1,\delta=2}\in\Z$ (by definition of $\eval$).

Hence $S\cup (S+1)\cup (S+2)\cup (S+3)=\Z$ and $\Z$ has a diophantine definition in $\M$ over our language. The similar definition for $\H$ results from the one for $\M$, by substituting each variable by a pair of variables, a ``numerator'' and a ``denominator'', by declaring that the denominators are not equal to $0$, and by clearing denominators in the resulting relations. This proves Theorem \ref{mainLogic} for $m=2$ variables. The general case (for any number of variables) is a trivial consequence of this. 

In the case $B=\C^m$, the condition $n\in \C$ may be substituted by $\exists c\  c^2=n^5-1$ (since non-singular curves of genus $\geq 2$ do not admit non-constant global meromorphic parametrisations --- for a proof, see for instance \cite[last section]{Pheidas95}).  

Moreover, over any field the statement $u\ne 0$ may be replaced by $\exists v\ u\cdot v=1$, while over $\H_{\bar z}(\C ^m)$ it can be substituted by 
$$
\exists \rho ,\tau \in \C(\tau |1\wedge z_1-\rho |u-\tau) ,
$$
where $|$ means ``divides'' and may be substituted as follows: $v_1|v_2$ if and only if $\exists v_3\ v_2=v_1\cdot v_3$.


Thanases Pheidas\\
Department of Mathematics\\
University of Crete-Heraklion, Greece\\

Xavier Vidaux\\
Departamento de Matem\'atica\\
Universidad de Concepci\'on\\
Avenida Esteban Iturra s/n\\
Concepci\'on, Chile\\

\end{document}